\documentclass[12pt,leqno]{amsart}
\usepackage{amsmath}
\usepackage{amssymb}
\def\overset#1#2{{\mathrel{\mathop {{#2}_{}}\limits^{#1}}}}
\def\underset#1#2{{\mathrel{\mathop {{}_{} {#2}}\limits_{{#1}_{}}}}}
\def\upplim_#1{\underset{#1}{\overline\lim}\;}
\def\lowlim_#1{\underset{#1}{\underline\lim}\;}
\newtheorem{corollary}[equation]{Corollary}

\newtheorem{lemma}[equation]{Lemma}
\newtheorem{proposition}[equation]{Proposition}

\newtheorem{theorem}[equation]{Theorem}

\newcommand{\C}{{\mathbb{C}}}

\renewcommand{\P}{{\mathbb{P}}}

\newcommand{\R}{{\mathbb{R}}}

\newcommand{\Z}{\mathbb{Z}}

\numberwithin{equation}{section}

\title[Degeneracy second main theorems for meromorphic mappings]{Degeneracy second main theorems for meromorphic mappings into projective varieties with hypersurfaces} 

\author{Si Duc Quang}

\begin{document}

\begin{abstract}
The purpose of this paper has twofold. The first is to establish a second main theorem with truncated counting functions for algebraically nondegenerate meromorphic mappings into an arbitrary projective variety intersecting a family of hypersurfaces in subgeneral position. 
%Let $f$ be an algebraically nondegenerate meromorphic mapping from a smooth projective subvariety $V\subset\P^n(\C)$ of dimension $k\ge 1$ and let $Q_1,...,Q_q$ be $q$ hypersurfaces in $\P^n(\C)$ of degree $d_i$, located in $N-$subgeneral position in $V$. We will prove that, for every $\epsilon >0$, 
%$$||\ (q-(N-k+1)(k+1)-\epsilon) T_f(r)\le\sum_{i=1}^q\frac{1}{d_i}N^{[M_0]}(r,f^*Q_i)+o(T_f(r)),$$
%where $M_0$ is a positive integer estimated explicitly. 
In our result, the truncation level of the counting functions is estimated explicitly. Our result is an extension of the classical second main theorem of H. Cartan, also is a generalization of the recent second main theorem of M. Ru and improves some recent results. The second purpose of this paper is to give another proof for the second main theorem for the special case where the projective variety is a projective space, by which the truncation level of the counting functions is estimated better than that of the general case.
\end{abstract}

\def\thefootnote{\empty}

\maketitle

\section{Introduction}

\vskip0.2cm 
Let $f$ be a meromorphic mapping from $\C^m$ into $\P^n(\C)$. For each hypersurface $Q$ in $\P^n(\C)$ with $f(\C^m)\not\subset Q$, we denote by $f^*Q$ the pull back divisor of $Q$ by $f$, where $Q$ is considered as a reduced divisor in $\P^n(\C)$. As usual, we denote by $T_f(r)$ the characteristic function of $f$ with respect to the hyperplane line bundle of $\P^n(\C)$ and $N^{[M]}_{Q(f)}(r)$ the counting function of $f^*Q$ with multiplicities truncated to level $M$ (see Section 2 for the definitions).
  
Let $\{H_i\}_{i=1}^q$ be hyperplanes of $\P^n(\C)$. Let $N\geq n$ and $q\geq N+1$.
We say that the family $\{H_j\}_{j=1}^q$ are in $N$-\textit{subgeneral position} if for every subset $R\subset \{1,2, \cdots, q\}$ with the cardinality $|R|=N+1$, 
$$\bigcap_{j\in R} H_j =\emptyset.  $$
If they are in $n$-subgeneral position, we simply say that they are in {\it general position}.

In 1933, H. Cartan \cite{Ca} established a second main theorem for linearly nondegenerate meromorphic mappings and hyperplanes as follows.

\vskip0.2cm
\noindent
\textbf{Theorem A.} {\it Let $f:  {\C}^m \to \P^n(\C)$ be a linearly nondegenerate meromorphic mapping and $\{H_i\}_{i=1}^q$ be hyperplanes in general position in $\P^n(\C)$. Then we have}
$$||\ \ (q-n-1)T_f(r) \leq \sum_{i=1}^q N^{[n]}(r,f^*H_i)+ o(T_f(r)).$$
Here, by the notation ``$\| P$'' we mean that the assertion $P$ holds for all $r\in [0,\infty)$ excluding a Borel subset $E$ of the interval $[0,\infty)$ with $\int_E dr<\infty$.

The above second main theorem of H. Cartan plays an important role in Nevanlinna theory, with many applications to Algebraic or Analytic geometry. We note that in Theorem A, the mapping $f$ is assumed to be linearly nondegenerate. In the case where $f$ may be linearly nondegenerate, we need to consider $f$ as a linearly nodegenerate mapping into the smallest subspace $V$ of $\P^n(\C)$ containing the image $f(\C^m)$. However in that case, the family of hyperplanes $\{H_i|_V\}_{i=1}^q$, the restrictions of $H_i'$s to $V$, may not be in general position in $V$, but they are in subgeneral position. Thanks the notion of Nochka weight, in 1983 Nochka \cite{Noc83} gave the following second main theorem for the case where the family of hyperplanes is in subgeneral position as follows.

\vskip0.2cm
\noindent
\textbf{Theorem B} (cf. \cite{Noc83,No05}). {\it Let $f:  {\C}^m \to \P^n(\C)$ be a meromorphic mapping and $\{H_i\}_{i=1}^q$ be hyperplanes in $N-$subgeneral position in $\P^n(\C)$ $(N\ge n)$. Then we have
$$||\ \ (q-2N+n-1)T_f(r) \leq \sum_{i=1}^q N^{[n]}(r,f^*H_i)+ o(T_f(r)).$$}

\noindent
The above result is usual called Cartan-Nochka's theorem. This result will deduce the second main theorem for the case of degenerate meromorphic mappings.

Over the last few decades, there have been several results generalizing this theorem to abstract objects. Especially, it is important to generalize these results to the case where the hyperplanes are replaced by hypersurfaces. In 2004, M. Ru \cite{Ru04} firstly proved a second main theorem for algebraically nondegenerate meromorphic mappings into $\mathbb P^{n}(\mathbb C)$ intersecting hypersurfaces in general position in $\mathbb P^{n}(\mathbb C)$. With the same assumptions, T. T. H. An and H. T. Phuong \cite{AP} improved the result of M. Ru by giving an explicit truncation level for counting functions. Similarly as above, in order to consider the case where the mappings may be algebraically degenerate, we need consider such mappings as algebraically nondegenerate mappings into the smallest subvariety of $\P^n(\C)$ containing thier images. In 2009, M. Ru \cite{Ru09} proved the following result for the case where the meromorphic mappings into an arbitrary projective subvariety of $\P^n(\C)$.

\vskip0.2cm
\noindent
\textbf{Theorem C} (cf. \cite{Ru09}) {\it Let $V\subset\P^n(\C)$ be a smooth complex projective variety of dimension $k\ge 1$. Let $Q_1,...,Q_q$ be $q$ hypersurfaces in
$\P^n(\C)$ of degree $d_i$, located in general position in $V$. Let $f:\C\to V$ be an algebraically nondegenerate holomorphic map. Then, for every $\epsilon >0$,
$$||\ (q-k-1-\epsilon) T_f(r)\le\sum_{i=1}^q\frac{1}{d_i}N_{Q_i(f)}(r)+o(T_f(r)).$$}

Here, the family of hypersurfaces $\{Q_i\}_{i=1}^q$ is said to be in $N-$subgeneral position in $V$ if for any $R\subset\{1,...,q\}$ with the cardinality $|R|=N+1$,
$$\bigcap_{j\in R} Q_j\cap V =\emptyset.$$
If they are in $k$-subgeneral position, we also say that they are in {\it general position} in $V$.

We may see that Theorem C also holds for the case where $f:\C^m\to V$ without so much modification in the proof. Also, we note that the family of hypersurfaces in the above result is assumed to be in general position in $V$ and the proof does not work if this family is in subgeneral position in $V$. In particular, Theorem C does not deduce the case of degenerate mappings. The difficult comes from the fact that there is no Nochka weights for the families of hypersurfaces. Hence, it may be very difficult to generalize this result for the case of degenerate mappings. 

Recently, there are some second main theorem for the case of degenerate maps given by many authors, such as Z. Chen, M. Ru and Q. Yan \cite{CRY01,CRY02}, L. Giang \cite{LG}, S. Lei and M. Ru \cite{LR} and others. However, these results are still not yet optimal. In particular they cannot deduce Theorem C of M. Ru or the classical second main theorem of H. Cartan. 

In another direction, recently D. D. Thai, D. P. An and S. D. Quang \cite{AQT,QA} gave a generalization of the notion of Nochka weight for the case of hypersurfaces in subgeneral position. Applying that new Nochka weight, these authors gave a generalization of Theorem B to the case of  meromorphic mappings with family of hypersurfaces in subgeneral position. They proved the following.

\vskip0.2cm
\noindent
\textbf{Theorem D} (cf. \cite{QA}) {\it Let $V$ be a complex projective subvariety of $\mathbb P^n(\mathbb C)$ of dimension $k\ (k\le n)$. Let $\{Q_i\}_{i=1}^q$ be hypersurfaces of $\mathbb P^n(\mathbb C)$ in $N$-subgeneral position with respect to $V$ with $\deg Q_i=d_i\ (1\le i\le q)$. Let $d$ be the least common multiple of $d_i'$s, i.e., $d=lcm (d_1,...,d_q)$. Let $f$ be a meromorphic mapping of $\mathbb C^m$ into $V$ such that $f$ is algebraically nondegenerate. Assume that  $q>\dfrac{(2N-k+1)H_{V}(d)}{k+1}.$ Then, we have
$$ \biggl |\biggl |\ \left (q-\dfrac{(2N-k+1)H_{V}(d)}{k+1}\right )T_f(r)\le \sum_{i=1}^{q}\dfrac{1}{d_i}N^{[H_{V}(d)-1]}_{Q_i(f)}(r)+o(T_f(r)),$$
where $H_V(d)$ is the Hilbert function of $V$.}

\vskip0.2cm 
Unfortunately, even this result generalizes the Cartan-Nochka's theorem, but in the case of  $d>1$, the total defect obtained is $\dfrac{(2N-k+1)H_{V}(d)}{k+1}$, which is more bigger than $(2N-k+1)$. The reason here is that the Nochka weight is still not generalized completely.

Our purpose in this paper is to generalize the above result of M. Ru and extend the result of H. Cartan to the case of meromorphic mappings into projective varieties with hypersurfaces in subgeneral position. We will consider the case where the mapping $f$ into a subvariety $V$ of dimension $k$ in $\P^n(\C)$ is algebraically nondegenerate and the family of $q$ hyperplanes $\{Q_i\}_{i=1}^q$ is in $N$-subgeneral position in $V$.  Our main idea to avoid using the Nochka weights is that: each time when we estimate the auxiliary functions, which deduces the second main theorem (see Lemma \ref{3.1} and (\ref{3.2}) for detail), we will replace $k+1$ hypersurfaces from $N+1$ hypersurfaces by $k+1$ new other hypersurfaces in general position in $V$ so that this process does not change the estimate. Moreover, in our result, we will give an explicit truncation level for the counting functions. Namely we will prove the following.

\begin{theorem}\label{1.1} Let $V\subset\P^n(\C)$ be a smooth complex projective variety of dimension $k\ge 1$. Let $Q_1,...,Q_q$ be hypersurfaces in $\P^n(\C)$ of degree $d_i$, located in $N-$subgeneral position in $V$. Let $d$ be the least common multiple of $d_1,...,d_q$, i.e., $d=lcm(d_1,...,d_q)$. Let $f:\C^m\to V$ be an algebraically nondegenerate meromorphic mapping. Then, for every $\epsilon >0$, 
$$\bigl |\bigl |\ \left (q-(N-k+1)(k+1)-\epsilon\right) T_f(r)\le\sum_{i=1}^q\frac{1}{d_i}N^{[M_0]}_{Q_i(f)}(r)+o(T_f(r)),$$
where $M_0=\left[\deg (V)^{k+1}e^kd^{k^2+k}p^k(2k+4)^kl^k\epsilon^{-k}\right]$ with $p=N-k+1$ and $l=(k+1)\cdot q!$
\end{theorem}
\noindent
Here, by the notation $[x]$ we denote the biggest integer which does not exceed the real number $x$.

Then, we see that, if the family of hypersurfaces is in general position, i.e. $N=k$, our result will imply Theorem C, and it also is an extension of the classical result of H. Cartan. Moreover,  since $(n-k+1)(k+1)\le (\frac{n}{2}+1)^2$ for every $1\le k\le n$, from the above theorem we immediately have the following corollary, which is a degeneracy second main theorem.

\begin{corollary}
Let $f:\C^m\to \P^n(\C)$ be a meromorphic mapping and let $Q_1,...,Q_q$ be hypersurfaces in $\P^n(\C)$ of degree $d_i$, located in general position. Then, for every $\epsilon >0$, there exists a positive integer $M$ such that
$$\biggl |\biggl |\ \left(q-\left (\frac{n}{2}+1\right)^2-\epsilon\right)T_f(r)\le\sum_{i=1}^q\frac{1}{d_i}N^{[M]}_{Q_i(f)}(r)+o(T_f(r)).$$
\end{corollary}

Moreover, for the case $V=\P^n(\C)$, we will give another proof for the second main theorem to get better truncation level for counting functions. Our last result in this paper is stated as follows.
\begin{theorem}\label{1.3}
Let $f:\C^m\to \P^n(\C)$ be an algebraically nondegenerate meromorphic mapping and let $Q_1,...,Q_q$ be $q$ hypersurfaces in $\P^n(\C)$ of degree $d_i$, located in $N$-subgeneral position. Let $d$ be the least common multiple of all $d_i'$s, i.e., $d=lcm (d_1,...,d_q)$. Then, for every $\epsilon >0$, 
$$||\ (q-p(n+1)-\epsilon)T_f(r)\le\sum_{i=1}^q\frac{1}{d_i}N^{[M_0]}_{Q_i(\tilde f)}(r)+o(T_f(r)),$$
where $p=N-n+1$ and $M_0=4\left (edp(n+1)^2I(\epsilon^{-1})\right )^n-1$. 
\end{theorem}
Here, by the notation $I(x)$ we denote the smallest integer number which is not smaller than the real number $x$. We see that the truncation level $M_0$  obtained in Theorem \ref{1.3} is much smaller than that in Theorem \ref{1.1}. Therefore, this result may will be more helpful to study problems which are related to truncating multiplicity of the intersections between meromorphic mappings and hypersurfaces. We would also like to emphasize that our above result implies the second main theorem of T. T. H. An and H. T. Phuong \cite{AP} for the case of $N=n$ with the a better truncation level.

{\bf Acknowledgements.} This research is funded by Vietnam National Foundation for Science and Technology Development (NAFOSTED) under grant number 101.04-2015.03.

\section{Basic notions and auxiliary results from Nevanlinna theory}

\noindent
{\bf 2.1 Characteristic function.}\ Set $||z|| = \big(|z_1|^2 + \dots + |z_m|^2\big)^{1/2}$ for
$z = (z_1,\dots,z_m) \in \C^m$ and define $B(r) := \{ z \in \C^m : ||z|| < r\},\quad S(r) := \{ z \in \C^m : ||z|| = r\}\ (0<r<\infty).$
Define
$$v_{m-1}(z) := \big(dd^c ||z||^2\big)^{m-1}\quad \quad \text{and}$$
$$\sigma_m(z):= d^c \log||z||^2 \land \big(dd^c\log||z||^2\big)^{m-1}
 \text{on} \quad \C^m \setminus \{0\}.$$

Let $f : \mathbb C^m \to \P^n(\C)$ be a meromorphic mapping. Let $\tilde f = ( f_0,\ldots, f_n)$ be a reduced representation of  $f,$  where $f_0,\ldots, f_n$ are holomorphic functions on $\mathbb C^m$ such that $I(f)=\{f_0=\cdots =f_n\}$ is  an analytic subset of codimension at least two of $\C^m$. The characteristic function of $f$ (with respect to the hyperplane line bundle of $\P^n(\C)$), denoted by $T_f (r)$, is defined by
$$ T_f (r)=\int_1^{r}\frac{dt}{t}\int_{B(t)}f^*\Omega\wedge v_{m-1},$$
where $B(t)=\{z\in\C^m\ ;\ ||z||<t\}$ and $\Omega$ is the Fubini-Study form on $\P^n(\C)$. By Jensen's formula, we have
$$ T_f(r)=\int_{S(r)}\log ||\tilde f||\sigma_m-\int_{S(1)}\log ||\tilde f||\sigma_m,$$
where $||\tilde f||=(|f_0|^2+\cdots |f_n|^2)^{\frac{1}{2}}$.

\vskip0.2cm 
\noindent
{\bf 2.2 Counting function.}\ Fix $(\omega_0:\cdots :\omega_n)$ be a homogeneous coordinate system on $\P^n(\C)$. Let $Q$ be a hypersurface in $\P^n(\C)$ of degree $d$. Thoughout this paper, we sometimes identify a hypersurface with the defining polynomial if there is no confusion.  Then we may write
$$ Q(\omega)=\sum_{I\in\mathcal T_d}a_I\omega^I, $$
where $\mathcal T_d=\{(i_0,...,i_n)\in\mathbb Z_{\ge 0}^{n+1}\ ;\ i_0+\cdots +i_n=d\}$, $\omega =(\omega_0,...,\omega_n)$, $\omega^I=\omega_0^{i_0}...\omega_n^{i_n}$ with $I=(i_0,...,i_n)\in\mathcal T_d$ and $a_I\ (I\in\mathcal T_d)$ are constants, not all zeros. Here $\Z_{\ge 0}$ denotes the set consist of all $(n+1)$-tuples of non negartive integers.  In the case $d=1$, we call $Q$ a hyperplane of $\P^n(\C)$.

Now for a divisor $\nu$ on $\mathbb C^m$ and for a positive integer $M$ or $M= \infty$, we set
$$\nu^{[M]}(z)=\min\ \{M,\nu(z)\},$$
and define
\begin{align*}
n(t) =
\begin{cases}
\int\limits_{|\nu|\,\cap B(t)}
\nu(z) \alpha^{m-1} & \text  { if } m \geq 2,\\
\sum\limits_{|z|\leq t} \nu (z) & \text { if }  m=1. 
\end{cases}
\end{align*}
Similarly, we define \quad $n^{[M]}(t).$ The counting function of $\nu$ is defined by
$$ N(r,\nu)=\int\limits_1^r \dfrac {n(t)}{t^{2m-1}}dt \quad (1<r<\infty).$$
Similarly, we define  \ $N(r,\nu^{[M]})$ and denote it by \ $N^{[M]}(r,\nu)$. 

If $\varphi$ is a nonzero meromorphic function on $\C^m$, we denote by $\nu^0_{\varphi}$ (resp. $\nu^\infty_\varphi$)the zero divisor (resp. pole divisor) of $\varphi$. We will denote by $N_{\varphi}(r)$ (resp. $N^{[M]}_{\varphi}(r)$) the counting function $N(r,\nu^0_\varphi)$ (resp. $N^{[M]}(r,\nu^0_\varphi)$).

\vskip0.2cm 
\noindent
{\bf 2.3 Proximity function.}\ Let $f$ and $Q$ be as above. The proximity function of $f$ with respect to $Q$, denoted by $m_f (r,Q)$, is defined by
$$m_f (r,Q)=\int_{S(r)}\log\frac{||\tilde f||^d}{|Q(\tilde f)|}\sigma_m-\int_{S(1)}\log\frac{||\tilde f||^d}{|Q(\tilde f)|}\sigma_m,$$
where $Q(\tilde f)=Q(f_0,...,f_n)$. This definition is independent of the choice of the reduced representation of $f$. 

\vskip0.2cm 
\noindent
{\bf 2.4 The first main theorem.}\ We denote by $f^*Q$ the pullback of the divisor $Q$ by $f$. We may see that $f^*Q$ identifies with the zero divisor $\nu^0_{Q(\tilde f)}$ of the function $Q(\tilde f)$. By Jensen's fomular, we have
$$N(r,f^*Q)=N_{Q(\tilde f)}(r)=\int_{S(r)}\log |Q(\tilde f)|\sigma_m-\int_{S(1)}\log |Q(\tilde f)| \sigma_m.$$
For convenience, we will denote $N_{Q(f)}(r)=N(r,f^*Q)$.

Then the first main theorem in Nevanlinna theory for meromorphic mappings and hypersurfaces is stated as follows.

\begin{theorem}[First Main Theorem] Let $f : \mathbb C^m \to \P^n(\C)$ be a holomorphic map, and let $Q$ be a hypersurface in $\P^n(\C)$ of degree $d$. If $f (\mathbb C^m) \not \subset Q$, then for every real number $r$ with $0 < r < +\infty$,
$$dT_f (r)=m_f (r,Q) + N_{Q(f)}(r)+O(1),$$
where $O(1)$ is a constant independent of $r$.
\end{theorem}

\vskip0.2cm 
\noindent
{\bf 2.5 Nevanlinna's functions.}\ Let $\varphi$ be a nonzero meromorphic function on $\mathbb C^m$, which is occasionally regarded as a meromorphic map into $\mathbb P^1(\mathbb C)$. The proximity function of $\varphi$ is defined by
$$m(r,\varphi)=\int_{S(r)}\log \max\ (|\varphi|,1)\sigma_m.$$
The Nevanlinna's characteristic function of $\varphi$ is define as follows
$$ T(r,\varphi)=N_{\frac{1}{\varphi}}(r)+m(r,\varphi). $$
Then 
$$T_\varphi (r)=T(r,\varphi)+O(1).$$
The function $\varphi$ is said to be small (with respect to $f$) if $||\ T_\varphi (r)=o(T_f(r))$.
%Here, by the notation ``$|| \ P$''  we mean the assertion $P$ holds for all $r \in [0,\infty)$ excluding a Borel subset $E$ of the interval $[0,\infty)$ with $\int_E dr<\infty$.

\vskip0.2cm 
\noindent
{\bf 2.6 Some lemmas.} The following play essential roles in Nevanlinna theory.

For each meromorphic function $g$ on $\C^m$ and  an $p$-tuples $\alpha =(\alpha_1,...,\alpha_m)\in\Z^m_{\ge 0}$, we set $|\alpha|=\sum_{i=1}^m\alpha_i$ and define 
$${\mathcal D}^{\alpha}g=\frac{\partial^{|\alpha|}g}{\partial^{\alpha_{1}}z_1...\partial^{\alpha_{m}}z_m}.$$
We have the lemma on logarithmic derivative stated as follows.
\begin{lemma}[{see \cite{NO}}]
Let $f$ be a nonzero meromorphic function on $\C.$ Then for all positive integer $k$,
$$\biggl|\biggl|\quad m\biggl(r,\dfrac{\mathcal D^{\alpha}(f)}{f}\biggl)=O(\log^+T(r,f))\ (\alpha\in \mathbb Z^m_+).$$
\end{lemma}

Repeating the argument in \cite[Proposition 4.5]{Fu}, we have the following proposition.
\begin{proposition}[{see \cite[Proposition 4.5]{Fu}}]\label{2.2}
Let $\Phi_1,\ldots,\Phi_{k+1}$ be meromorphic functions on $\C^m$ such that $\{\Phi_1,\ldots,\Phi_{k+1}\}$ 
are  linearly independent over $\C.$
Then  there exists an admissible set $\{\alpha_i=(\alpha_{i1},\ldots,\alpha_{im})\}_{i=1}^{k+1} \subset \Z^m_+$ with $|\alpha_i|=\sum_{j=1}^{m}\alpha_{ij}\le i-1 \ (1\le i \le k+1)$ satisfying the following two properties:

(i)\  $\{{\mathcal D}^{\alpha_i}\Phi_1,\ldots,{\mathcal D}^{\alpha_i}\Phi_{k+1}\}_{i=1}^{k+1}$ is linearly independent over $\mathcal M$ (the field of all meromorphic functions on $\C^m$),\ i.e., 
$$\det{({\mathcal D}^{\alpha_i}\Phi_j)}_{1\le i,j\le k+1}\not\equiv 0,$$

(ii) $\det \bigl({\mathcal D}^{\alpha_i}(h\Phi_j)\bigl)=h^{k+1}\det \bigl({\mathcal D}^{\alpha_i}\Phi_j\bigl)$ for
any nonzero meromorphic function $h$ on $\C^m.$
\end{proposition}
We note that $\alpha_1,...,\alpha_{k+1}$ are chosen uniquely in an explicit way (see \cite{Fu}). Then we define the general Wronskian of the mapping $\Phi =(\Phi_1,\ldots,\Phi_{k+1})$ by 
$$ W(\Phi):=\det{({\mathcal D}^{\alpha_i}\Phi_j)}_{1\le i,j\le k+1}. $$

The following theorem is called the general form of the second main theorem, which is given by M. Ru \cite{Ru97}
\begin{theorem}[see \cite{Ru97}]\label{2.4}
Let $\C^m\rightarrow\P^n(\C)$ be a linearly nondegenerate holomorphic curve with a reduced representation $\tilde f=(f_0,...,f_n)$. Let $H_1,...,H_q$ be arbitrary hyperplanes in $\P^n(\C)$. Denote by $W(\tilde f)$ the general Wronskian of $(f_0,...,f_n)$.
Then, for every $\epsilon >0$,
$$\int\limits_{S(r)}\max_{K}\log\prod_{i\in K}\dfrac{||\tilde f||\cdot ||H_i||}{|H_i(\tilde f)|}\sigma_m+ N_{W(\tilde f)}(r)\le (n +1+\epsilon)T_f(r),$$
where the maximum is taken over all subsets $K$ of $\{1,...,q\}$ so that $\{H_i\ |\ i\in K\}$ is linearly independent.
\end{theorem}

We note that, M. Ru proved the above theorem for the case $m=1$, but this theorem also holds for the general case. The proof of this theorem for the general case is the same with that of the special case $m=1$.

\vskip0.2cm 
\noindent
{\bf 2.6 Chow weights and Hilbert weights.} We recall the notion of Chow weights and Hilbert weights from \cite{Ru09}.

Let $X\subset\P^n(\C)$ be a projective variety of dimension $k$ and degree $\Delta$.  To $X$ we associate, up to a constant scalar, a unique polynomial
$$F_X(\textbf{u}_0,\ldots,\textbf{u}_k) = F_X(u_{00},\ldots,u_{0n};\ldots; u_{k0},\ldots,u_{kn})$$
in $k+1$ blocks of variables $\textbf{u}_i=(u_{i0},\ldots,u_{in}), i = 0,\ldots,k$, which is called the Chow form of $X$, with the following
properties: $F_X$ is irreducible in $\C[u_{00},\ldots,u_{kn}]$, $F_X$ is homogeneous of degree $\Delta$ in each block $\textbf{u}_i, i=0,\ldots,k$, and $F_X(\textbf{u}_0,\ldots,\textbf{u}_k) = 0$ if and only if $X\cap H_{\textbf{u}_0}\cap H_{\textbf{u}_k}\ne\emptyset$, where $H_{\textbf{u}_i}, i = 0,\ldots,k$, are the hyperplanes given by
$$u_{i0}x_0+\cdots+ u_{in}x_n=0.$$

Let $F_X$ be the Chow form associated to $X$. Let ${\bf c}=(c_0,\ldots, c_n)$ be a tuple of real numbers. Let $t$ be an auxiliary variable. We consider the decomposition
\begin{align}\label{2.5}
\begin{split}
F_X(t^{c_0}u_{00},&\ldots,t^{c_n}u_{0n};\ldots ; t^{c_0}u_{k0},\ldots,t^{c_n}u_{kn})\\ 
& = t^{e_0}G_0(\textbf{u}_0,\ldots,\textbf{u}_n)+\cdots +t^{e_r}G_r(\textbf{u}_0,\ldots, \textbf{u}_n).
\end{split}
\end{align}
with $G_0,\ldots,G_r\in\C[u_{00},\ldots,u_{0n};\ldots; u_{k0},\ldots,u_{kn}]$ and $e_0>e_1>\cdots>e_r$. The Chow weight of $X$ with respect to $c$ is defined by
\begin{align}\label{2.6}
e_X({\bf c}):=e_0.
\end{align}
For each subset $J = \{j_0,...,j_k\}$ of $\{0,...,n\}$ with $j_0<j_1<\cdots<j_k,$ we define the bracket
\begin{align}\label{2.7}
[J] = [J]({\bf u}_0,\ldots,{\bf u}_n):= \det (u_{ij_t}), i,t=0,\ldots ,k,
\end{align}
where $\textbf{u}_i = (u_{i_0},\ldots ,u_{in})$ denotes the blocks of $n+1$ variables. Let $J_1,\ldots ,J_\beta$ with $\beta=\binom{n+1}{k+1}$ be all subsets of $\{0,...,n\}$ of cardinality $k+1$.

Then the Chow form $F_X$ of $X$ can be written as a homogeneous polynomial of degree $\Delta$ in $[J_1],\ldots,[J_\beta]$. We may see that for $\textbf{c}=(c_0,\ldots,c_n)\in\R^{n+1}$ and for any $J$ among $J_1,\ldots,J_\beta$,
\begin{align}\label{2.8}
\begin{split}
[J](t^{c_0}u_{00},\ldots,t^{c_n}u_{0n},&\ldots,t^{c_0}u_{k0},\ldots,t^{c_n}u_{kn})\\
&=t\sum_{j\in J}c_j[J](u_{00},\ldots,u_{0n},\ldots,u_{k0},\ldots,u_{kn}).
\end{split}
\end{align}

For $\textbf{a} = (a_0,\ldots,a_n)\in\mathbb Z^{n+1}$ we write ${\bf x}^{\bf a}$ for the monomial $x^{a_0}_0\cdots x^{a_n}_n$. Let $I=I_X$ be the prime ideal in $\C[x_0,\ldots,x_n]$ deffning $X$. Let $\C[x_0,\ldots,x_n]_m$ denote the vector space of homogeneous polynomials in $\C[x_0,\ldots,x_n]$ of degree $m$ (including $0$). Put $I_m :=\ C[x_0,\ldots,x_n]_m\cap I$ and define the Hilbert function $H_X$ of $X$ by, for $m = 1, 2,...,$
\begin{align}\label{2.9}
H_X(m):=\dim (\C[x_0,...,x_n]_m/I_m).
\end{align}
By the usual theory of Hilbert polynomials,
\begin{align}\label{2.10}
H_X(m)=\Delta\cdot\frac{m^n}{n!}+O(m^{n-1}).
\end{align}
The $m$-th Hilbert weight $S_X(m,{\bf c})$ of $X$ with respect to the tuple ${\bf c}=(c_0,\ldots,c_n)\in\mathbb R^{n+1}$ is defined by
\begin{align}\label{2.11}
S_X(m,{\bf c}):=\max\left (\sum_{i=1}^{H_X(m)}{\bf a}_i\cdot{\bf c}\right),
\end{align}
where the maximum is taken over all sets of monomials ${\bf x}^{{\bf a}_1},\ldots,{\bf x}^{{\bf a}_{H_X(m)}}$ whose residue classes modulo $I$ form a basis of $\C[x_0,\ldots,x_n]_m/I_m.$

%According to Mumford [Mu, Prop. 2.11], we have
%$$ S_X(m,{\bf c})=e_X({\bf c})\cdot\frac{m^{n+1}}{(n+1)!}+O(m^n).$$
%Together with (\ref{2.6}), this implies that
%\begin{align}\label{2.8}
%\lim\limits_{m\to +\infty}\frac{1}{mH_X(m)}\cdot S_X(m,{\bf c})=\frac{1}{(n+1)\Delta}\cdot e_X({\bf c}).
%\end{align}
%We call $\frac{1}{mH_X(m)}S_X(m,{\bf c})$ the $m$-th normalized Hilbert weight and $\frac{1}{(n+1)\Delta}e_X({\bf c})$ the normalized Chow weight of $X$ with respect to ${\bf c}$.

The following theorem is due to J. Evertse and R. Ferretti \cite{EF1} and is restated by M. Ru \cite{Ru09} for the special case when the filed $K=\C$.
\begin{theorem}[{Theorem 4.1 \cite{EF1}, see also Theorem 2.1 \cite{Ru09}}]\label{2.12}
Let $X\subset\P^n(\C)$ be an algebraic variety of dimension $k$ and degree $\Delta$. Let $m>\Delta$ be an integer and let ${\bf c}=(c_0,\ldots,c_n)\in\mathbb R^{n+1}_{\geqslant 0}$.
Then
$$ \frac{1}{mH_X(m)}S_X(m,{\bf c})\ge\frac{1}{(n+1)\Delta}e_X({\bf c})-\frac{(2n+1)\Delta}{m}\cdot\left (\max_{i=0,...,n}c_i\right). $$
\end{theorem}

The following lemma is due to J. Evertse and R. Ferretti \cite{EF2} for the case of the field $\mathbb Q^p$ and reproved by M. Ru \cite{Ru09} for the case of the field $\C$.
\begin{lemma}[{Lemma 3.2 \cite{Ru09}, see also Lemma 5.1 \cite{EF2}}]\label{2.13}
Let $Y$ be a subvariety of $\P^{q-1}(\C)$ of dimension $n$ and degree $\Delta$. Let ${\bf c}=(c_1,\ldots, c_q)$ be a tuple of positive reals. Let $\{i_0,...,i_n\}$ be a subset of $\{1,...,q\}$ such that
$$Y \cap \{y_{i_0}=\cdots =y_{i_n}=0\}=\emptyset.$$
Then
$$e_Y({\bf c})\ge (c_{i_0}+\cdots +c_{i_n})\Delta.$$
\end{lemma}

\section{Proof of main theorems}

In order to prove main theorems, we first give the following key lemma.

\begin{lemma}\label{3.1}
Let $V$ be a smooth projective subvariety of $\P^n(\C)$ of dimension $k$. Let $Q_1,...,Q_{N+1}$ be hypersurfaces in $\P^n(\C)$ of the same degree $d\ge 1$, such that
$$\left (\bigcap_{i=1}^{N+1}Q_i\right )\cap V=\emptyset.$$
Then there exists $k$ hypersurfaces $P_{2},...,P_{k+1}$ of the forms
$$P_t=\sum_{j=2}^{N-k+t}c_{tj}Q_j, \ c_{tj}\in\C,\ t=2,...,k+1,$$
such that $\left (\bigcap_{t=1}^{k+1}P_t\right )\cap V=\emptyset,$ where $P_1=Q_1$.
\end{lemma}
\begin{proof} Set $P_1=Q_1$. It is easy to see that
$$ \dim \left(\bigcap_{i=1}^tQ_i\right)\cap V\le N-t,\ t=N-k+2,...,N+1,$$
where $\dim\emptyset =-\infty$.

Step 1. Firstly, we will construct $P_2$ as follows. For each irreducible component $\Gamma$ of dimension $k-1$ of $P_1\cap V$, we put 
$$V_{1\Gamma}=\{c=(c_2,...,c_{N-k+2})\in\C^{N-k+1}\ ;\ \Gamma\subset Q_c,\text{ where }Q_c=\sum_{j=2}^{N-k+2}c_jQ_j\}.$$
Here, we also consider the case where $Q_c$ may be zero polynomial and it detemines all $\P^n(\C)$. It easy to see that $V_{1\Gamma}$ is a subspace of $\C^{N-k+1}$. Since $\dim \left(\bigcap_{i=0}^{N-k+1}Q_i\right)\cap V\le k-2$, there exists $i \ (1\le i\le N-k+1)$ such that $\Gamma\not\subset Q_i$. This implies that $V_{1\Gamma}$ is a proper subspace of $\C^{N-k+1}$. Since the set of irreducible components of dimension $k-1$ of $P_0\cap V$ is at most countable, 
$$ \C^{N-k+1}\setminus\bigcup_{\Gamma}V_{1\Gamma}\ne\emptyset. $$
Hence, there exists $(c_{12},...,c_{1(N-k+2)})\in\C^{N-k+1}$ such that
$$ \Gamma\not\subset P_2$$
for all irreducible components of dimension $k-1$ of $Q_1\cap V$, where
$P_2=\sum_{j=2}^{N-k+2}c_{1j}Q_j.$
This clearly implies that $\dim \left(P_1\cap P_2\right)\cap V\le k-2.$

Step 2. Now, for each irreducible component $\Gamma'$ of dimension $k-2$ of $\left(P_1\cap P_2\right)\cap V$, put 
$$V_{2\Gamma'}=\{c=(c_2,...,c_{N-k+3})\in\C^{N-k+2}\ ;\ \Gamma\subset Q_c,\text{ where }Q_c=\sum_{j=2}^{N-k+3}c_jQ_j\}.$$
Hence, $V_{2\Gamma'}$ is a subspace of $\C^{N-k+2}$. Since $\dim \left(\bigcap_{i=1}^{N-k+3}Q_i\right)\cap V\le k-3$, there exists $i, (2\le i\le N-k+3)$ such that $\Gamma'\not\subset Q_i$. This implies that $V_{2\Gamma'}$ is a proper subspace of $\C^{N-k+2}$. Since the set of irreducible components of dimension $k-2$ of $\left(P_1\cap P_2\right)\cap V$ is at most countable, 
$$ \C^{N-k+2}\setminus\bigcup_{\Gamma'}V_{2\Gamma'}\ne\emptyset. $$
Then, there exists $(c_{22},...,c_{2(N-k+3)})\in\C^{N-k+2}$ such that
$$ \Gamma'\not\subset P_3 $$
for all irreducible components of dimension $k-2$ of $P_1\cap P_2\cap V$, where
$P_3=\sum_{j=2}^{N-k+3}c_{2j}Q_j.$
And hence, this implies that $\dim \left(P_1\cap P_2\cap P_3\right)\cap V\le k-3.$

Repeating again the above steps, after the $k^{\rm th}$-step we get hypersurfaces $P_2,...,P_{k+1}$ stisfying
$$ \dim\left(\bigcap_{j=1}^tP_j\right)\cap V\le k-t. $$
In particular, $\left(\bigcap_{j=1}^{k+1}P_j\right)\cap V=\emptyset.$ We complete the proof of the lemma.
\end{proof}

\begin{proof}[{\sc Proof of Theorem \ref{1.1}}]
Firstly, we will prove the theorem for the case where all hypersurfaces $Q_i\ (1\le i\le q)$ are of the same degree $d$. We may also assume that $q>(N-k+1)(k+1)$. 
We denote by $\mathcal I$ the set of all permutations of the set $\{1,....,q\}$. Denote by $n_0$ the cardinality of $\mathcal I$. Then we have $n_0=q!$, and we may write that
$\mathcal I=\{I_1,....,I_{n_0}\}$
where $I_i=(I_i(1),...,I_i(q))\in\mathbb N^q$ and $I_1<I_2<\cdots <I_q$ in the lexicographic order.

For each $I_i\in\mathcal I$, we denote by $P_{i,1},...,P_{i,{k+1}}$ the hypersurfaces obtained in Lemma \ref{3.1} with respect to the hypersurfaces $Q_{I_i(1)},...,Q_{I_i(n+1)}$. It is easy to see that there exists a positive constant $B\ge 1$, which is chosen common for all $I_i\in\mathcal I$, such that
$$ |P_{i,t}(\omega)|\le B\max_{1\le j\le N-k+t}|Q_{I_i(j)}(\omega)|, $$
for all $1\le t\le k+1$ and for all $\omega=(\omega_0,...,\omega_n)\in\C^{n+1}$. 

Consider a reduced representation $\tilde f=(f_0,\ldots ,f_n): \C^m\rightarrow \C^{n+1}$ of $f$. Fix an element $I_i\in\mathcal I$. We denote by $S(i)$ the set of all points $z\in \C\setminus\bigcup_{i=1}^qQ_i(\tilde f)^{-1}(\{0\})$ such that
$$ |Q_{I_i(1)}(\tilde f)(z)|\le |Q_{I_i(2)}(\tilde f)(z)|\le\cdots\le |Q_{I_i(q)}(\tilde f)(z)|.$$
Since $Q_{1},\ldots,Q_{q}$ are in $N-$subgeneral position in $V$, by Lemma \ref{3.2}, there exist a positive constant $A$, which is chosen common for all $I_i$, such that
$$ ||\tilde f (z)||^d\le A\max_{1\le j\le N+1}|Q_{I_i(j)}(\tilde f)(z)|\ \forall z\in S(i). $$
Therefore, for $z\in S(i)$ we have
\begin{align*}
\prod_{i=1}^q\dfrac{||\tilde f (z)||^d}{|Q_i(\tilde f)(z)|}&\le A^{q-N-1}\prod_{j=1}^{N+1}\dfrac{||\tilde f (z)||^d}{|Q_{I_i(j)}(\tilde f)(z)|}\\
&\le A^{q-N-1}B^{k}\dfrac{||\tilde f (z)||^{(N+1)d}}{\bigl (\prod_{j=1}^{N+1-k}|Q_{I_i(j)}(\tilde f)(z)|\bigl )\cdot\prod_{j=2}^{k+1}|P_{i,j}(\tilde f)(z)|}\\
&\le A^{q-N-1}B^{k}\dfrac{||\tilde f (z)||^{(N+1)d}}{|P_{i,1}(\tilde f)(z)|^{N-k+1}\cdot\prod_{j=2}^{k+1}|P_{i,j}(\tilde f)(z)|}\\
&\le A^{q-N-1}B^{k}C^{(N-k)}\dfrac{||\tilde f (z)||^{(N+1)d+(N-k)kd}}{\prod_{j=1}^{k+1}|P_{i,j}(\tilde f)(z)|^{N-k+1}},
\end{align*}
where $C$ is a positive constant, which is chosen common for all $I_i\in\mathcal I$, such that 
$$|P_{i,j}(\omega)|\le C||\omega||^d, \ \forall \omega=(\omega_0,...,\omega_n)\in\C^{n+1}.$$
The above inequality implies that
\begin{align}\label{3.2}
\log \prod_{i=1}^q\dfrac{||\tilde f (z)||^d}{|Q_i(\tilde f)(z)|}\le \log(A^{q-N-1}B^{k}C^{(N-k)})+(N-k+1)\log \dfrac{||\tilde f (z)||^{(k+1)d}}{\prod_{j=1}^{k+1}|P_{i,j}(\tilde f)(z)|}.
\end{align}

We consider the mapping $\Phi$ from $V$ into $\P^{l-1}(\C)\ (l=n_0(k+1))$, which maps a point $x\in V$ into the point $\Phi(x)\in\P^{l-1}(\C)$ given by
$$\Phi(x)=(P_{1,1}(x):\cdots : P_{1,k+1}(x):P_{2,1}(x):\cdots:P_{2,k+1}(x):\cdots:P_{n_0,1}(x):\cdots :P_{n_0,k+1}(x)).$$ 
Let $Y=\Phi (V)$. Since $V\cap\bigcap_{j=1}^{k+1}P_{1,j}=\emptyset$, $\Phi$ is a finite morphism on $V$ and $Y$ is a complex projective subvariety of $\P^{p-1}(\C)$ with $\dim Y=k$ and $\Delta:=\deg Y=\le d^k.\deg V$. 
For every 
$${\bf a} = (a_{1,1},\ldots ,a_{1,k+1},a_{2,1}\ldots,a_{2,k+1},\ldots,a_{n_0,1},\ldots,a_{n_0,k+1})\in\mathbb Z^l_{\ge 0}$$ 
and
$${\bf y} = (y_{1,1},\ldots ,y_{1,k+1},y_{2,1}\ldots,y_{2,k+1},\ldots,y_{n_0,1},\ldots,y_{n_0,k+1})$$ 
we denote ${\bf y}^{\bf a} = y_{1,1}^{a_{1,1}}\ldots y_{1,k+1}^{a_{1,k+1}}\ldots y_{n_0,1}^{a_{n_0,1}}\ldots y_{n_0,k+1}^{a_{n_0,k+1}}$. Let $u$ be a positive integer. We set
\begin{align}\label{3.3}
n_u:=H_Y(u)-1,\ l_u:=\binom{l+u-1}{u}-1,
\end{align}
and define the space
$$ Y_u=\C[y_1,\ldots,y_p]_u/(I_Y)_u, $$
which is a vector space of dimension $n_u+1$. We fix a basis $\{v_0,\ldots, v_{n_u}\}$ of $Y_u$ and consider the meromorphic mapping $F$ with a reduced representation
$$ \tilde F=(v_0(\Phi\circ \tilde f),\ldots ,v_{n_u}(\Phi\circ \tilde f)):\C^m\rightarrow \C^{n_u+1}. $$
Hence $F$ is linearly nondegenerate, since $f$ is algebraically nodegenerate.

Now, we fix an index $i\in\{1,....,n_0\}$ and a point $z\in S(i)$. We define 
$${\bf c}_z = (c_{1,1,z},\ldots,c_{1,k+1,z},c_{2,1,z},\ldots,c_{2,k+1,z},\ldots,c_{n_0,1,z},\ldots,c_{n_0,k+1,z})\in\mathbb Z^{p},$$ 
where
\begin{align}\label{3.4}
c_{i,j,z}:=\log\frac{||\tilde f(z)||^d||P_{i,j}||}{|P_{i,j}(\tilde f)(z)|}\text{ for } i=1,...,n_0 \text{ and }j=1,...,k+1.
\end{align}
We see that $c_{i,j,z}\ge 0$ for all $i$ and $j$. By the definition of the Hilbert weight, there are ${\bf a}_{1,z},...,{\bf a}_{H_Y(u),z}\in\mathbb N^{l}$ with
$$ {\bf a}_{i,z}=(a_{i,1,1,z},\ldots,a_{i,1,k+1,z},\ldots,a_{i,n_0,1,z},\ldots,a_{i,n_0,k+1,z}), a_{i,j,s,z}\in\{1,...,l_u\}, $$
 such that the residue classes modulo $(I_Y)_u$ of ${\bf y}^{{\bf a}_{1,z}},...,{\bf y}^{{\bf a}_{H_Y(u),z}}$ form a basic of $\C[y_1,...,y_p]_u/(I_Y)_u$ and
\begin{align}\label{3.5}
S_Y(u,{\bf c}_z)=\sum_{i=1}^{H_Y(u)}{\bf a}_{i,z}\cdot{\bf c}_z.
\end{align}
We see that ${\bf y}^{{\bf a}_{i,z}}\in Y_m$ (modulo $(I_Y)_m$). Then we may write
$$ {\bf y}^{{\bf a}_{i,z}}=L_{i,z}(v_0,\ldots ,v_{H_Y(u)}), $$ 
where $L_{i,z}\ (1\le i\le H_Y(u))$ are independent linear forms.
We have
\begin{align*}
\log\prod_{i=1}^{H_Y(u)} |L_{i,z}(\tilde F(z))|&=\log\prod_{i=1}^{H_Y(u)}\prod_{\overset{1\le t\le n_0}{1\le j\le k+1}}|P_{tj}(\tilde f(z))|^{a_{i,t,j,z}}\\
&=-S_Y(m,{\bf c}_z)+duH_Y(u)\log ||\tilde f(z)|| +O(uH_Y(u)).
\end{align*}
This implies that
\begin{align*}
\log\prod_{i=1}^{H_Y(u)}\dfrac{||\tilde F(z)||\cdot ||L_{i,z}||}{|L_{i,z}(\tilde F(z))|}=&S_Y(u,{\bf c}_z)-duH_Y(u)\log ||\tilde f(z)|| \\
&+H_Y(u)\log ||\tilde F(z)||+O(uH_Y(u)).
\end{align*}
Here we note that $L_{i,z}$ depends on $i$ and $z$, but the number of these linear forms is finite. We denote by $\mathcal L$ the set of all $L_{i,z}$ occuring in the above inequalities. Then we have
\begin{align}\label{3.6}
\begin{split}
S_Y(u,{\bf c}_z)\le&\max_{\mathcal J\subset\mathcal L}\log\prod_{L\in \mathcal J}\dfrac{||\tilde F(z)||\cdot ||L||}{|L(\tilde F(z))|}+duH_Y(u)\log ||\tilde f(z)||\\
& -H_Y(u)\log ||\tilde F(z)||+O(uH_Y(u)),
\end{split}
\end{align}
where the maximum is taken over all subsets $\mathcal J\subset\mathcal L$ with $\sharp\mathcal J=H_Y(u)$ and $\{L;L\in\mathcal J\}$ is linearly independent.
From Theorem \ref{2.12} we have
\begin{align}\label{3.7}
\dfrac{1}{uH_Y(u)}S_Y(u,{\bf c}_z)\ge&\frac{1}{(k+1)\Delta}e_Y({\bf c}_z)-\frac{(2k+1)\Delta}{u}\max_{\underset{1\le i\le n_0}{1\le j\le k+1}}c_{i,j,z}
\end{align}
We chose an index $i_0$ such that $z\in S(i_0)$. It is clear that
\begin{align*}
\max_{\underset{1\le i\le n_0}{1\le j\le k+1}}c_{i,j,z}\le \sum_{1\le j\le k+1}\log\frac{||\tilde f(z)||^d||P_{i_0,j}||}{|P_{i_0,j}(\tilde f)(z)|}+O(1),
\end{align*}
where the term $O(1)$ does not depend on $z$ and $i_0$.
Combining (\ref{3.6}), (\ref{3.7}) and the above remark, we get
\begin{align}\nonumber
\frac{1}{(k+1)\Delta}e_Y({\bf c}_z)\le &\dfrac{1}{uH_Y(u)}\left (\max_{\mathcal J\subset\mathcal L}\log\prod_{L\in \mathcal J}\dfrac{||\tilde F(z)||\cdot ||L||}{|L(\tilde F(z))|}-H_Y(u)\log ||\tilde F(z)||\right )\\
\label{3.8}
\begin{split}
&+d\log ||\tilde f(z)||+\frac{(2n+1)\Delta}{m}\max_{\underset{1\le i\le n_0}{1\le j\le k+1}}c_{i,j,z}+O(1/m)\\
\le &\dfrac{1}{uH_Y(u)}\left (\max_{\mathcal J\subset\mathcal L}\prod_{L\in\mathcal J}\dfrac{||\tilde F(z)||\cdot ||L||}{|L(\tilde F(z))|}-H_Y(u)\log ||\tilde F(z)||\right )\\
&+d\log ||\tilde f(z)||+\frac{(2n+1)\Delta}{m}\sum_{1\le j\le k}\log\frac{||\tilde f(z)||^d||P_{i_0,j}||}{|P_{i_0,j}(\tilde f)(z)|}+O(1/m).
\end{split}
\end{align}
Since $P_{i_0,1}...,P_{i_0,k+1}$ are in general with respect to $X$, By Lemma \ref{2.13}, we have
\begin{align}\label{3.9}
e_Y({\bf c}_z)\ge (c_{i_0,1,z}+\cdots +c_{i_0,k+1,z})\cdot\Delta =\left (\sum_{1\le j\le k}\log\frac{||\tilde f(z)||^d||P_{i_0,j}||}{|P_{i_0,j}(\tilde f)(z)|}\right )\cdot\Delta.
\end{align}
Then, from (\ref{3.2}), (\ref{3.8}) and (\ref{3.9}) we have
\begin{align}\label{3.10}
\begin{split}
\frac{1}{N-k+1}&\log \prod_{i=1}^q\dfrac{||\tilde f (z)||^d}{|Q_i(\tilde f)(z)|}\le\dfrac{k+1}{uH_Y(u)}\left (\max_{\mathcal J\subset\mathcal L}\log\prod_{L\in\mathcal J}\dfrac{||\tilde F(z)||\cdot ||L||}{|L(\tilde F(z))|}-H_Y(u)\log ||\tilde F(z)||\right )\\
&+d(k+1)\log ||\tilde f(z)||+\frac{(2k+1)(k+1)\Delta}{u}\sum_{\underset{1\le i\le n_0}{1\le j\le k+1}}\log\frac{||\tilde f(z)||^d||P_{i,j}||}{|P_{i,j}(\tilde f)(z)|}+O(1),
\end{split}
\end{align}
where the term $O(1)$ does not depend on $z$. 

Integrating both sides of the obove inequality, we obtain 
\begin{align}\nonumber
\frac{1}{N-k+1}\sum_{i=1}^qm_f(r,Q_i)\le& \dfrac{k+1}{uH_Y(u)}\left (\int\limits_{S(r)}\max_{\mathcal J\subset\mathcal L}\log\prod_{L\in\mathcal J}\dfrac{||\tilde F(z)||\cdot ||L||}{|L(\tilde F(z))|}\sigma_m-H_Y(u)T_F(r)\right )\\
\label{3.11}
&+d(k+1)T_f(r)+\frac{(2k+1)(k+1)\Delta}{u}\sum_{{\underset{1\le i\le n_0}{1\le j\le k+1}}}m_f(r,P_{i,j}).
\end{align} 
On the other hand, by Theorem \ref{2.4}, for every $\epsilon'>0$ (which will be chosen later) we have
\begin{align*}
\bigl |\bigl |\ \int\limits_{S(r)}\max_{\mathcal J\subset\mathcal L}&\log\prod_{L\in\mathcal J}\dfrac{||\tilde F(z)||\cdot ||L||}{|L(\tilde F(z))|}\sigma_m-H_Y(u)T_F(r)\\
&\le -N_{W(\tilde F)}(r)+\epsilon' T_F(r)\le  -N_{W(\tilde F)}(r)+\epsilon' du T_f(r).
\end{align*}
Combining this inequality with (\ref{3.11}), we have
\begin{align}\label{3.12}
\begin{split}
\bigl (q-(N-k+1)&(k+1)\bigl )T_f(r)\\
\le& \sum_{i=1}^q\frac{1}{d}N_{Q_i(f)}(r)-\dfrac{(N-k+1)(k+1)}{duH_Y(u)}(N_{W(\tilde F)}(r)-\epsilon' du T_f(r))\\
&+\frac{(N-k+1)(2k+1)(k+1)\Delta}{ud}\sum_{{\underset{1\le i\le n_0}{1\le j\le k+1}}}m_f(r,P_{i,j})
\end{split}
\end{align}

We now estimate the quantity $N_{W(\tilde F)}(r)$. We consider a point $z\in\C^m$ which is outside the indeterminacy locus of $f$. Without loss of generality, we may assume that $z\in S(1)$, where $I_1=(1,....,q)$. Then we see that $\nu^0_{Q_i(f)}(z)=0$ for all $i\ge N+1$, since $\{Q_1,...,Q_q\}$ is in $N$-subgeneral position in $V$. We set $c_{i,j}=\max\{0,\nu^0_{P_{i,j}}(z)-H_Y(u)\}$ and 
$${\bf c}=(c_{1,1},\ldots,c_{1,k+1},\ldots,c_{n_0,1},\ldots,c_{n_0,k+1})\in\mathbb Z^l_{\ge 0}.$$
Then there are 
$${\bf a}_i=(a_{i,1,1},\ldots,a_{i,1,k+1},\ldots,a_{i,n_0,1},\ldots,a_{i,n_0,k+1}),a_{i,j,s}\in\{1,...,l_u\}$$
such that ${\bf y}^{{\bf a}_1},...,{\bf y}^{{\bf a}_{H_Y(u)}}$ is a basic of $Y_u$ and
$$ S_Y(m,{\bf c})=\sum_{i=1}^{H_Y(u)}{\bf a}_i{\bf c}.$$
Similarly as above, we write ${\bf y}^{{\bf a}_i}=L_i(v_1,...,v_{H_Y(u)})$, where $L_1,...,L_{H_Y(u)}$ are independent linear forms in variables $y_{i,j}\ (1\le i\le n_0,1\le j\le k+1)$. By the property of the general Wronskian, we see that 
$$W(\tilde F)=cW(L_1(\tilde F),...,L_{H_Y(u)}(\tilde F)),$$
where $c$ is a nonzero constant. This yields that
$$ \nu^0_{W(\tilde F)}(z)=\nu^0_{W(L_1(\tilde f),...,L_{H_Y(u)}(\tilde F))}\ge\sum_{i=1}^{H_Y(u)}\max\{0,\nu^0_{L_i(\tilde F)}(z)-n_u\}$$
We also easily see that
$$ \nu^0_{L_i(\tilde F)}(z)=\sum_{\underset{1\le j\le n_0}{1\le s\le k+1}}a_{i,j,s}\nu^0_{P_{j,s}(\tilde f)}(z), $$
and hence
$$ \max\{0,\nu^0_{L_i(\tilde F)}(z)-n_u\}\ge\sum_{i=1}^{H_Y(u)}a_{i,j,s}c_{j,s}={{\bf a}_i}\cdot{\bf c}. $$
Thus, we have
\begin{align}\label{3.13}
 \nu^0_{W(\tilde F)}(z)\ge\sum_{i=1}^{H_Y(u)}{{\bf a}_i}\cdot{\bf c}=S_Y(u,{\bf c}).
\end{align}
Since $P_{1,1},...,P_{1,k+1}$ are in general position, then by Lemma \ref{2.13} we have
$$ e_Y({\bf c})\ge \Delta\cdot\sum_{j=1}^{k+1}c_{1,j}=\Delta\cdot\sum_{j=1}^{k+1}\max\{0,\nu^0_{P_{1,j}(f)}(z)-n_u\}. $$
On the other hand, by Theorem \ref{2.12} we have that 
\begin{align*}
 S_Y(u,{\bf c}) &\ge\frac{uH_Y(u)}{(k+1)\Delta}e_Y({\bf c})-(2k+1)\Delta H_Y(u)\max_{\underset{1\le i\le n_0}{1\le j\le k+1}}c_{i,j}\\
&\ge \frac{uH_Y(u)}{k+1}\sum_{j=1}^{k+1}\max\{0,\nu^0_{P_{1,j}(\tilde f)}(z)-n_u\}-(2k+1)\Delta H_Y(u)\max_{\underset{1\le i\le n_0}{1\le j\le k+1}}\nu^0_{P_{i,j}(\tilde f)}(z).
\end{align*}
Combining this inequality and (\ref{3.13}), we have
\begin{align}\label{3.14}
\begin{split}
\dfrac{k+1}{duH_Y(u)}\nu^0_{W(\tilde F)}(z)\ge&\frac{1}{d}\sum_{j=1}^{k+1}\max\{0,\nu^0_{P_{1,j}(\tilde f)}(z)-n_u\}\\
&-\frac{(2k+1)(k+1)\Delta}{du}\max_{\underset{1\le i\le n_0}{1\le j\le k+1}}\nu^0_{P_{i,j}(\tilde f)}(z).
\end{split}
\end{align}
Also it is easy to see that $\nu^0_{P_{1,j}(\tilde f)}(z)\ge \nu^0_{Q_{N-k+j}(\tilde f)}(z)$ for all $2\le j\le k+1$. Therefore, we have
\begin{align*}
(N-k+1)&\sum_{j=1}^{k+1}\max\{0,\nu^0_{P_{1,j}(\tilde f)}(z)-n_u\}\\
&\ge (N-k+1)\max\{0,\nu^0_{Q_{1}(\tilde f)}(z)-n_u\}+\sum_{j=2}^{k+1}\max\{0,\nu^0_{Q_{N-k+j}(\tilde f)}(z)-n_u\}\\
&\ge\sum_{i=1}^{N+1}\max\{0,\nu^0_{Q_i(\tilde f)}(z)-n_u\}=\sum_{i=1}^{q}\max\{0,\nu^0_{Q_i(\tilde f)}(z)-n_u\}.
\end{align*}
Combining this inequality and (\ref{3.14}), we have
\begin{align*}
\dfrac{(N-k+1)(k+1)}{duH_Y(u)}&\nu^0_{W(\tilde F)}(z)\ge\frac{1}{d}\sum_{i=1}^{q}\max\{0,\nu^0_{Q_i(\tilde f)}(z)-n_u\}\\
&-\frac{(N-k+1)(2k+1)(k+1)\Delta}{du}\max_{\underset{1\le i\le n_0}{1\le j\le k+1}}\nu^0_{P_{i,j}(\tilde f)}(z)\\
\ge& \frac{1}{d}\sum_{i=1}^{q}(\nu^0_{Q_i(\tilde f)}(z)-\min\{\nu^0_{Q_i(\tilde f)}(z),u\})\\
&-\frac{(N-k+1)(2k+1)(k+1)\Delta}{du}\max_{\underset{1\le i\le n_0}{1\le j\le k+1}}\nu^0_{P_{i,j}(\tilde f)}(z).
\end{align*}
Integrating both sides of this inequality, we obtain 
\begin{align}\label{3.15}
\begin{split}
\dfrac{(N-k+1)(k+1)}{duH_Y(u)}N_{W(\tilde F)}(r)&\ge\frac{1}{d}\sum_{i=1}^{q}(N_{Q_i(f)}(r)-N^{[n_u]}_{Q_i(f)}(r))\\
&-\frac{(N-k+1)(2k+1)(k+1)\Delta}{du}\max_{\underset{1\le i\le n_0}{1\le j\le k+1}}N_{P_{i,j}(f)}(r).
\end{split}
\end{align}

Combining inequalities (\ref{3.12}) and (\ref{3.15}), we get
\begin{align}\label{3.16}
\begin{split}
\bigl |\bigl |\ (q-&(N-k+1)(k+1)\bigl )T_f(r)\le \sum_{i=1}^q\frac{1}{d}N^{[n_u]}_{Q_i(f)}(r)+\dfrac{(N-k+1)(k+1)}{H_Y(u)}\epsilon' T_f(r))\\
&+\frac{(N-k+1)(2k+1)(k+1)\Delta}{ud}\sum_{{\underset{1\le i\le n_0}{1\le j\le k+1}}}(N_{P_{i,j}(f)}(r)+m_f(r,P_{i,j}))\\
=&\sum_{i=1}^q\frac{1}{d}N^{[n_u]}_{Q_i(f)}(r)+\left (\dfrac{(N-k+1)(k+1)\epsilon'}{H_Y(u)}+\frac{(N-k+1)(2k+1)(k+1)p\Delta}{u}\right)T_f(r).
\end{split}
\end{align}

We now choose $u$ is the biggest integer such that
$$ u< (N-k+1)(2k+1)(k+1)p\Delta\epsilon^{-1} $$
and 
$$\epsilon'=\epsilon-\frac{(N-k+1)(2k+1)(k+1)p\Delta}{u}>0.$$
Form (\ref{3.16}), we have
\begin{align}\label{3.17}
\bigl |\bigl |\ (q-(N-k+1)(k+1)-\epsilon\bigl )T_f(r)\le \sum_{i=1}^q\frac{1}{d}N^{[n_u]}_{Q_i(f)}(r).
\end{align}
We note that $\deg Y=\Delta\le d^k\deg (V)$. Then the number $n_u$ is estimated as follows
\begin{align*}
n_u&=H_Y(u)-1\le\Delta \binom{k+u}{k}\le d^k\deg (V)e^{k}\left(1+\frac{u}{k}\right)^{k}\\
&<d^k\deg (V)e^k\left ((N-k+1)(2k+4)p\Delta\epsilon^{-1}\right)^k\\
&\le \left[\deg (V)^{k+1}e^kd^{k^2+k}(N-k+1)^k(2k+4)^kp^k\epsilon^{-k}\right]=M_0.
\end{align*} 
Then, the theorem is proved for the case where all hypersurfaces $Q_i$ have the same degree.

Now, for the general case where $Q_i$ is of the degree $d_i$, $i=1,...,q,$ and $d=lcm (d_1,...,d_q)$. Then the hypersurfaces $Q_i^{d/d_i}\ (1\le i\le q)$ are of the same degree $d$. Applying the above result, we have
\begin{align*}
\bigl |\bigl |\ (q-(N-k+1)(k+1)-\epsilon)T_f(r)&\le\sum_{i=1}^{q}\frac{1}{d}N^{[M_0]}_{Q^{d/d_i}(f)}(r)\\
&\le\sum_{i=1}^{q}\frac{1}{d_i}N^{[M_0]}_{Q^{d/d_i}(f)}(r)+o(T_f(r)).
\end{align*}
This completed the proof of the theorem.
\end{proof}

\begin{proof}[{\sc Proof of Theorem \ref{1.3}}]
Similar as the argument of the proof of Theorem \ref{1.1}, we may assume that all $Q_i\ (1\le i\le q)$ are of the same degree $d$ and $q>(N-k+1)(k+1)$.
Letting $V=\P^n(\C)$, $k=n$ and using the the same notation, we will repeat some arguments as in the proof Theorem \ref{1.1}. Similarly as (\ref{3.14}), for $z\in S(i_0)$ we have
\begin{align}\label{3.18}
\begin{split}
\log \prod_{i=1}^q\dfrac{||\tilde f (z)||^d}{|Q_i(\tilde f)(z)|}&\le \log(A^{q-N-1}B^{k}C^{(N-k)})+(N-k+1)\log \dfrac{||\tilde f (z)||^{(k+1)d}}{\prod_{j=1}^{k+1}|P_{i_0,j}(\tilde f)(z)|}\\
&\le (N-k+1)\log \dfrac{||\tilde f (z)||^{kd}}{\prod_{j=1}^{k}|P_{i_0,j}(\tilde f)(z)|}+O(1),
\end{split}
\end{align}
where the term $O(1)$ does not depend on $z$.

Now, for a positive integer $u$, we denote by $V_u$ the vector subspace of $\C[x_0,\ldots, x_n]$ which consists of  all homogeneous polynomials of degree $u$ and zero polynomial. We consider $u$ divisible by $d$. For each $(i)=(i_1,\ldots,i_n)\in\mathbf N_0^n$ with $\sigma (i)=\sum_{s=1}^ni_s\le\frac{N}{d}$, we set
$$W^{i_0}_{(i)}=\sum_{(j)=(j_1,\ldots ,j_k)\ge (i)}P_{i_0,1}^{j_1}\cdots P_{i_0,k}^{j_k}\cdot V_{u-d\sigma (j)}.$$
Then we see that $W^{i_0}_{(0,\ldots,0)}=V_u$ and $W^{i_0}_{(i)}\supset W^{i_0}_{(j)}$ if $(i)<(j)$ (in the sense of lexicographic order). Therefore, $W^{i_0}_{(i)}$ is a filtration of $V_u$. 
We have the following lemma due to Corvaja and Zannier in \cite{CZ}.
\begin{lemma}[see \cite{CZ}]\label{3.19}
Let $(i)=(i_1,\ldots ,i_k),(i)'=(i_1',\ldots ,i_k')\in \mathbf N^k_0$. Suppose that $(i')$ follows $(i)$ in the lexicographic ordering and defined
$$ m^{i_0}_{(i)}=\dim \dfrac{W^{i_0}_{(i)}}{W^{i_0}_{(i)'}}.$$
Then, we have $m^{i_0}_{(i)}=d^k,$ provided $d \sigma (i)<u-kd$. 
\end{lemma}
We assume that 
$$ V_u=W^{i_0}_{(i)_1}\supset W^{i_0}_{(i)_2}\supset\cdots\supset W^{i_0}_{(i)_K}, $$
where $(i)_s=(i_{1s},...,i_{us})$, $W^{i_0}_{(i)_{s+1}}$ follows $W^{i_0}_{(i)_s}$ in the ordering and $(i)_K=(\frac{u}{d},0,\ldots ,0)$. We see that $K$ is the number of $k$-tuples $(i_1,\ldots,i_k)$ with $i_j\ge 0$ and $i_1+\cdots +i_k\le\frac{k}{d}$. Then we easily estimate that
$$ K =\binom{\frac{u}{d}+n}{n}.$$
We define $m^{i_0}_s=\dim\frac{W^{i_0}_{(i)_s}}{W^{i_0}_{(i)_{s+1}}}$ for all $s=1,\ldots, K-1$ and set $m^{i_0}_K=1$. 

Let $m_u=\dim V_u$. From the above filtration, we may choose a basis $\{\psi^{i_0}_1,\ldots,\psi^{i_0}_{m_u}\}$ of $V_u$ such that  
$$\{\psi_{m_u-(m^{i_0}_s+\cdots +m^{i_0}_K)+1},\ldots ,\psi_{m_u}\}$$
 is a basis of $W^{i_0}_{(i)_s}$. For each $s\in\{1,\ldots,K\}$ and $l\in\{m_u-(m^{i_0}_s+\cdots +m^{i_0}_K)+1,\ldots, m_u-(m^{i_0}_{s+1}+\cdots +m^{i_0}_K)\}$, we may write
$$ \psi^{i_0}_l=P_{i_0,1}^{i_{1s}}\ldots P_{i_0,n}^{i_{ns}}h_l,\ \text{ where } (i_{1s},\ldots,i_{ks})=(i)_s, h_l\in W^{i_0}_{u-d\sigma (i)_s}. $$
Then we have
\begin{align*}
|\psi^{i_0}_l(\tilde f)(z)|&\le |P_{i_0,1} (\tilde f)(z)|^{i_{1s}}\ldots |P_{i_0,n} (\tilde f)(z)|^{i_{ks}}|h_l(\tilde f)(z)|\\
& \le c_2|P_{i_0,1} (\tilde f)(z)|^{i_{1s}}\ldots |P_{i_0,n} (\tilde f)(z)|^{i_{ks}}||\tilde f(z)||^{u-d\sigma(i)_s}\\
&=c_2\left (\dfrac{|P_{i_0,1} (\tilde f)(z)|}{||\tilde f(z)||^d}\right)^{i_{1s}}\ldots\left (\dfrac{|P_{i_0,n} (\tilde f)(z)|}{||\tilde f(z)||^d}\right)^{i_{ks}}||\tilde f (z)||^u,
\end{align*} 
where $c_2$ is a positive constant independently from $l$, $i_0$, $f$ and $z$. This implies that
\begin{align}\label{3.20}
\begin{split}
\log\prod_{l=1}^{m_u}|\psi^{i_0}_l(\tilde f)(z)|&\le\sum_{s=1}^Km^{i_0}_s\left (i_{1s}\log\dfrac{|P_{i_0,1} (\tilde f)(z)|}{||\tilde f(z)||^d}+\cdots+i_{ns}\log\dfrac{|P_{i_0,n} (\tilde f)(z)|}{||\tilde f(z)||^d}\right)\\ 
& \ \ \ +m_uu\log ||\tilde f (z)||+\log c_2.
\end{split}
\end{align}

We fix $\phi_1,...,\phi_{m_u}$, a basic of $V_u$, $\psi_s^{i_0}(\tilde f)=L_s^{i_0}(\tilde F)$, where $L_s^{i_0}\ (1\le s\le m_u)$ are linear forms and $\tilde F=(\phi_1(\tilde f),\ldots ,\phi_{u}(\tilde f))$ is a reduced representation of a meromorphic mapping $F$. We set
\begin{align*}
b_j^{i_0}=\sum_{s=1}^{K}m_s^{i_0}i_{js},\ 1\le j\le k.
\end{align*}
From (\ref{3.20}) we have that
$$\log\prod_{s=1}^{m_u}|L_{s}^{i_0}(\tilde F)(z)|\le\log\left(\prod_{j=1}^{n}\biggl (\dfrac{|P_{i_0,j}(\tilde f)(z)|}{||\tilde f(z)||^d}\biggl )^{b_j^{i_0}}\right)+m_uu\log ||\tilde f(z)||+\log c_2.$$
We set $b=\min_{j,i_0}b_j^{i_0}$. Since $f$ is algebraically non degenerate over $\C$, $F$ is linearly non degenerate over $\C$. Then there exists an admissible set $\alpha =(\alpha_1,...,\alpha_{m_u})\in (\mathbb Z_+^m)^u$, with $|\alpha_s|\le s-1$, such that the general Wronskian satisfying
\begin{align*}
W(\Phi(\tilde f)):=W^{\alpha}(\Phi(\tilde f))=W^{\alpha}(\phi_1(\tilde f),....,\phi_{m_u}(\tilde f))=\det (\mathcal D^{\alpha_i}(\phi_s(\tilde f)))_{1\le i,s\le m_u}\not\equiv 0,
\end{align*}
where $\Phi =(\phi_1,...,\phi_{m_u})$.
We also have
\begin{align}\label{3.21}
\begin{split}
\log \frac{||\tilde f(z)||^{qdb}|W^{\alpha}(\Phi(\tilde f))(z)|^p}{\prod_{i=1}^{q}|Q_i(\tilde f)(z)|^b}&\le \log\frac{||\tilde f(z)||^{pndb}|W^\alpha(\Phi(\tilde f))(z)|^p}{\prod_{j=1}^n|P_{i_0,j}(\tilde f)(z)|^{pb}}+O(1)\\
&\le\log \frac{||\tilde f(z)||^{pd\sum_{j=1}^{n}b_j^{i_0}}|W^{\alpha}(\Phi(\tilde f))(z)|^p}{\prod_{j=1}^{n}|P_{i_0,j}(\tilde f)(z)|^{pb^{i_0}_j}} +O(1)\\
&\le \log \frac{||\tilde f(z)||^{pm_uu}|W^{\alpha}(\Phi(\tilde f))(z)|^p}{\prod_{i=1}^{m_u}|\psi_i^{i_0}(\tilde f)(z)|^p}+O(1)\\
&\le \log \frac{||\tilde f(z)||^{pm_uu}|W^{\alpha}(\Psi^{i_0}(\tilde f))(z)|^p}{\prod_{i=1}^{u}|\psi_i^{i_0}(\tilde f)(z)|^p}+O(1),
\end{split}
\end{align}
where $\Psi^{i_0}=(\psi^{i_0},...,\psi^{i_0}_{m_u})$, $W^{\alpha}(\Psi^{i_0}(\tilde f))=\det (\mathcal D^{\alpha_i}(\psi^{i_0}_{s}(\tilde f)))_{1\le i,s\le m_u}$ and the term $O(1)$ depends only on $u$ and $\{Q_i\}_{i=1}^{q}$.
This inequality implies that
\begin{align}\label{3.22}
\log \frac{||\tilde f(z)||^{qdb-pm_uu}|W^{\alpha}(\Phi(\tilde f))(z)|^p}{(\prod_{i=1}^{q}|Q_i(\tilde f)(z)|^b)}\le\log\frac{|W^{\alpha}(\Phi(\tilde f))(z)|^p}{\prod_{i=1}^{m_u}|\psi_i^{i_0}(\tilde f)(z)|^p}+O(1),
\end{align}
for all $z\in\C^m$ outside a proper analytic subset of $\C^m$, which is the union of zero sets of functions $Q_i(\tilde f), P_{i_0,j}(\tilde f)$.

From the above inequality, we obtain
\begin{align}\label{3.23}
\log \frac{||\tilde f(z)||^{qdb-pm_uu}|W^{\alpha}(\Phi(\tilde f))(z)|^p}{(\prod_{i=1}^{q}|Q_i(\tilde f)(z)|^b)}\le\sum_{i_0=1}^{n_0}\log^{+} \frac{|W^{\alpha}(\Phi(\tilde f))(z)|^p}{\prod_{i=1}^{m_u}|\psi_i^{i_0}(\tilde f)(z)|^p}+O(1),
\end{align}
for all $z\in\C^m$ outside the union of zero sets of functions $Q_i(\tilde f)$ and $P_i(\tilde f)$.

We easily see that $W^{\alpha}(\Phi(\tilde f))(z)=C_{i_0}W^{\alpha}(\Psi^{i_0}(\tilde f))$, where
$$ W^{\alpha}(\Psi^{i_0}(\tilde f))=\det (\mathcal D^{\alpha_i}(\psi^{i_0}_{s}(\tilde f)))_{1\le i,s\le m_u} $$
and $C_{i_0}$ is a nonzero constant. Then by using the lemma on logarithmic derivative, we easily have
$$\biggl |\biggl |\ \int_{S(r)}\log^+\frac{|W^{\alpha}(\Phi(\tilde f))(z)|^p}{\prod_{i=1}^{m_u}|\psi_i^{i_0}(\tilde f)(z)|^p}\sigma_m=\int_{S(r)}\log^+\frac{|W^{\alpha}(\Psi^{i_0}(\tilde f))(z)|^p}{\prod_{i=1}^{m_u}|\psi_i^{i_0}(\tilde f)(z)|^p}\sigma_m =o(T_f(r)).$$
Integrating both sides of (\ref{3.23}) over $S(r)$ with the help of the above inequality, we obtain
$$ ||\ (qdb-pm_uu)T_f(r)+pN_{W^{\alpha}(\Phi(\tilde f))}(r)-b\sum_{i=1}^qN_{Q_i(f)}(r)\le o(T_f(r)),$$
which means that
\begin{align}\label{3.24}
 ||\ (q-\frac{pm_uu}{db})T_f(r)\le\sum_{i=1}^q\dfrac{1}{d}N_{Q_i(f)}(r)-\frac{p}{db}N_{W^{\alpha}(\Phi(\tilde f))}(r)+o(T_f(r)). 
\end{align}

We now estimate 
$\bigl (\sum_{j=1}^q\nu^0_{Q_j(f)}-\frac{p}{b}\nu^0_{W^{\alpha}(\Phi(\tilde f))}\bigl )$. 
Fix $z\in \C^m$. Without loss of generality, we may assume that
$$\nu^0_{Q_1(\tilde f)}(z)\ge\cdots
\ge\nu^0_{Q_{t}(\tilde f)}(z)>0=\nu^0_{Q_{t+1}(\tilde f)}(z)=\cdots =\nu^0_{Q_{q}(\tilde f)}(z),$$
where $0\le t\le N,$ ($t$ may be zero). We note that $I_1=(1,2,...,q)$ and let $M=m_u-1$. Then we will see that
\begin{align*}
\nu^0_{P_{1,1}(\tilde f)}(z)&=\nu^0_{Q_1(\tilde f)}(z),\\ 
\nu^0_{P_{1,j}(\tilde f)}(z)&\ge\nu^0_{Q_{N-n+j}(\tilde f)}(z) \ (2\le j\le n). 
\end{align*}
 We have
$$p\nu^0_{W^{\alpha}(\Phi(\tilde f))}(z)=p\nu^0_{W(\Psi^{1}(\tilde f))}(z)\ge p\sum_{s=1}^{m_u}\max\{\nu^0_{\psi_s^{1}(\tilde f)}(z)-M,0\}.$$
For $\psi=P_{1,1}^{t_1}...P_{1,n}^{t_n}h\in\{\psi_{s}^{1}\}_{s=1}^{m_u}$, we have
$$\psi (\tilde f)(z)=P_{1,1}^{t_1}(\tilde f)(z)\ldots P_{1,n}^{t_n}(\tilde f)(z).h(\tilde f)(z).$$
Hence
\begin{align*}
\max \{ \nu^0_{\psi (\tilde f)}(z)-M,0\}
\ge & \sum_{j=1}^{n}\max \{\nu^0_{(P_{1,j}^{t_j}(\tilde f)}(z)-M,0 \}\\
\ge & \sum_{j=1}^{n}t_j\max \{\nu^0_{P_{1,j}(\tilde f)}(z)-M,0\}.
\end{align*}
This implies that
\begin{align*}
p\sum_{s=1}^{m_u}\max&\{\nu^0_{\psi_s^{1}(\tilde f)}(z)-M,0\}\ge p\sum_{(i)}m_{(i)}^{1}\sum_{j=1}^{n}t_j\max\{\nu^0_{P_{1,j}(\tilde f)}(z)-M,0\}\\
&=p\sum_{j=1}^{n}b_j^{1}\max\{\nu^0_{P_{1,j}(\tilde f)}(z)-M,0\}\ge p\sum_{j=1}^nb\max\{\nu^0_{P_{1,j}(\tilde f)}(z)-M,0\}\\
&\ge\sum_{j=1}^{N}b\max\{\nu^0_{Q_j(\tilde f)}(z)-M,0\}=\sum_{i=1}^{q}b\max\{\nu^0_{Q_i(\tilde f)}(z)-M,0\}.
\end{align*}
Hence 
\begin{align*}
\sum_{i=1}^{q}\nu^0_{Q_i(\tilde f)}(z)-\frac{p}{b}\nu^0_{W^{\alpha}(\Phi(\tilde f))}(z)&\le\sum_{i=1}^{q}\max\{\nu^0_{Q_i(\tilde f)}(z)-m_u+1,0\}.
\end{align*}
This implies that
\begin{align*}
\sum_{i=1}^qN_{Q_i(f)}(r)-\frac{p}{b}N_{W^{\alpha}(\Phi(\tilde f))}(r)\le\sum_{i=1}^{q}N^{[m_u-1]}_{Q_i(f)}(r).
\end{align*}
Combining (\ref{3.23}) and the above inequality, we have
\begin{align}\label{3.25}
||(qdb-\frac{pm_uu}{db})T_f(r)\le\sum_{i=1}^{q}\frac{1}{d}N^{[m_u-1]}_{Q_i(f)}(r)+o(T_f(r)).
\end{align}

We now have some estimates. First, 
$$m_u=\binom{u+n}{n}=\frac{(u+1)\cdots (u+n)}{1\cdots n}.$$
Second, since the number of nonnegative integer $l$-tuples with summation $\le T$ is equal to the number of nonnegative integer $(l+1)$-tuples with summation exactly equal $T\in \Z$, which is $\bigl (^{T+l}_{\ \ l}\bigl )$ and since the sum below is independent of $j$, we have that
\begin{align*}
b_{j}^{1}=&\sum_{\sigma (i)\le u/d}m^{1}_{(i)}i_{j}\ge\sum_{\sigma (i)\le u/d-n}m^{1}_{(i)}i_{j}\\ 
=&\sum_{\sigma (i)\le u/d-n}d^ni_j=\frac{d^n}{n+1}\sum_{\sigma (i)\le u/d-n}\sum_{j=1}^{n+1}i_j\\
=&\frac{d^n}{n+1}\sum_{\sigma (i)\le u/d-n}\left (\frac{u}{d}-n\right)=\frac{d^n(u-nd)}{(n+1)d}\binom{u/d}{n}=\frac{d^nu(u/d-1)\cdots (u/d-n)}{1\cdots (n+1)d}.
\end{align*}
This implies that 
\begin{align*}
\frac{pm_uu}{db}\le p(n+1)\frac{(u+1)\cdots (u+n)}{(u-d)\cdots (u-nd)}
\le & p(n+1)\prod_{j=1}^n\frac{u+j}{u-(n+1)d+jd}\\
\le & p(n+1)\left (\frac{u}{u-(n+1)d}\right )^n.
\end{align*}

We now note that, for each positive number $x\in (0,\frac{1}{(n+1)^2}]$, we have
\begin{align}\label{3.26}
\begin{split}
(1+x)^{n}&=1+nx+\sum_{i=2}^n\binom{n}{i}x^{i}\le 1+nx+\sum_{i=1}^2\frac{n^{i}}{i!(n+1)^{2i-2}}x\\
&=1+nx+\sum_{i=2}^n\frac{1}{i!}x\le 1+(n+1)x.
\end{split}
\end{align}
We chose $u= (n+1)d+p(n+1)^3I(\epsilon^{-1})d$. Then $u$ is divisible by $d$ and one gets 
\begin{align}\label{3.27}
\frac{(n+1)d}{u-(n+1)d}=\frac{(n+1)d}{p(n+1)^3I(\epsilon^{-1})d}\le \frac{1}{(n+1)^2}.
\end{align}
Therefore, using (\ref{3.26}) and (\ref{3.27}) we have
\begin{align*}
\frac{pm_uu}{db}\le & p(n+1)\left (\frac{u}{u-(n+1)d}\right )^n= p(n+1)\left (1+\frac{(n+1)d}{u-(n+1)d}\right )^n\\
\le &p(n+1)\left (1+(n+1)\frac{(n+1)d}{u-(n+1)d}\right)\\
\le &p(n+1)\left (1+(n+1)\frac{(n+1)d}{p(n+1)^3I(\epsilon^{-1})d}\right )\\
\le &p(n+1)\left (1+\frac{1}{p(n+1)\epsilon^{-1}}\right )=p(n+1)+\epsilon.
\end{align*}
Also, the number $m_u$ is estimated as follows
\begin{align*}
m_u&=\binom{u+n}{n}\le e^n\left (1+\frac{u}{n}\right )^n\le e^n\left (\frac{n+(n+1)d}{n}+\frac{p(n+1)^3I(\epsilon^{-1})d}{n}\right )^n\\
&=e^n(p(n+1)^2I(\epsilon^{-1})d)^n\left (1+\frac{1}{n}+\frac{n+(n+1)d}{np(n+1)^2I(\epsilon^{-1})d}\right)^n\\
&\le \left (edp(n+1)^2I(\epsilon^{-1})\right )^n\cdot\left (1+\frac{1}{n}+\frac{2}{n(n+1)}\right)^n< 4\left (edp(n+1)^2I(\epsilon^{-1})\right )^n=M_0+1.
\end{align*}
Thus, from (\ref{3.24}) we obtain
\begin{align*}
||(q-(N-n+1)(n+1)-\epsilon)T_f(r)\le\sum_{i=1}^{q}\frac{1}{d}N^{[M_0]}_{Q_i(f)}(r)+o(T_f(r)).
\end{align*}
Then we complete the proof of the theorem.
\end{proof}

\vskip0.2cm
{\footnotesize 
\noindent
{\sc Si Duc Quang}\\
Department of Mathematics, Hanoi National University of Education,\\
136-Xuan Thuy, Cau Giay, Hanoi, Vietnam.\\
%$^2$ Thang Long Institute of Mathematics and Applied Sciences,\\
%Nghiem Xuan Yem, Hoang Mai, HaNoi, Vietnam.\\
\textit{E-mail}: quangsd@hnue.edu.vn}

\end{document}